\documentclass[11pt]{amsart}
\usepackage{amsmath}
\usepackage{amsthm}
\usepackage{amssymb}
\usepackage{amsfonts}
\usepackage{latexsym}
\usepackage{graphicx}
\usepackage{avant}
\usepackage{epsfig}
\usepackage{pstricks}
\usepackage{pst-node}
\usepackage{multido}
\usepackage[all]{xy}
\usepackage{tikz}
\usetikzlibrary{arrows}
\usepackage{verbatim}

 \newtheorem{thm}{Theorem}[section]
 
 \newtheorem{lem}[thm]{Lemma}
 \newtheorem{prop}[thm]{Proposition}
 \theoremstyle{definition}

 \theoremstyle{remark}
 \newtheorem{rem}[thm]{Remark}
 \numberwithin{equation}{section}

 \newcommand{\R}{\mathcal{R}}

\newcommand{\E}{\mbox{E}}

\title{Bandwidth of the product of paths of the same length}
\author{Louis J. Billera and Sa\'ul A. Blanco}
\thanks{Supported, in part, by NSF grant DMS-0555268.}
\address{Department of Mathematics, Cornell University, Ithaca, NY, 14853}
\email{billera@math.cornell.edu}
\address{Department of Mathematical Sciences, DePaul University, Chicago, IL, 60614}
\email{sblancor@depaul.edu}
\begin{document}

\begin{abstract}
In this note we give a numerical expression for the bandwidth $bw(P_{n}^{d})$ of the $d$-product of a path with $n$ edges, $P_{n}^{d}$. We prove that this bandwidth is given by the sum of certain multinomial coefficients. We also show that $bw(P_{n}^{d})$ is bounded above and below by the largest coefficient in the expansion of $(1+x+\cdots+x^{n})^{k}$, with $k\in\{d,d+1\}$. Moreover, we compare the asymptotic behavior of  $bw(P_{n}^{d})$ with the bandwidth of the labeling obtained by ordering the vertices of $P_{n}^{d}$ in lexicographic order. 
\end{abstract}

\maketitle

\section{Introduction}

Given a graph $G=(V,E)$, a \emph{labeling} $f$ of $G$ is a bijection, $f:V\to[n]$, where $n=|V|$. The \emph{bandwidth} $bw(f)$ of $f$ is defined as
\[
bw(f)=\max_{(u,v)\in E}|f(u)-f(v)|.
\]

The \emph{bandwidth} of $G$, denoted by $bw(G)$, is the minimum bandwidth over all possible labelings of $G$, that is,
\[
bw(G)=\min_{f}bw(f).
\] 
The \emph{bandwidth problem} consists of finding the bandwidth of a given graph. Lai and Williams~\cite{LK99} provided an end-of-the-century survey containing some of the main results known at the time.

 In this note we give a numerical expression for the bandwidth of the $d$-product of a path with $n$ edges, $P_{n}^{d}$. The vertices of $P_{n}^{d}$ are labeled with elements of the set $\{0,1,\ldots,n\}^{d}$, and two vertices $u$ and $v$ are connected by and edge if and only if $u$ and $v$ differ only at one component in which they have absolute difference one (see Figure~\ref{fig:hales_3^2} for an illustration of $P_{2}^{2}$).

One of the classical results on the field is the following explicit formula for the bandwidth of the hypercube $Q_{d}=P_{1}^{d}$ of dimension $d>0$:
\begin{equation}\label{eq:hypercube}
bw(Q_{d})=\sum_{i=0}^{d-1}\binom{i}{\lfloor\frac{i}{2}\rfloor}
\end{equation}
This result first appeared, without proof, in~\cite{H66} (cf. \cite[Corollary 4.4]{H04}). The first appearance in the literature of a proof for~(\ref{eq:hypercube}) was given by~\cite{WWD09} and utilizes the recursive nature of the Hales order, which we describe in the Section~\ref{sec:P_n^d}.

In his pioneering paper \cite{H66}, Harper shows  that the labeling given by the Hales order, defined in the following section,  solves the bandwidth problem for $Q_{n}$. His result was later generalized by Moghadam~\cite{M05}, who showed that the Hales order, under the natural modification of the weight function, also solves the bandwidth problem for graphs that are product of paths of any length.

In this note we provide a formula for $bw(P_{n}^{d})$ in terms of sums of coefficients in the expansion of $(1+x+x^{2}+\cdots+x^{n})^{d}$. In particular, the central multinomial coefficients provide a lower and upper bound for $bw(P_{n}^{d})$. Furthermore, we show that $bw(P_{n}^{d})$ is $o(bw(L_{n}^{d}))$ as $d\to\infty$, where $L_{n}^{d}$ is the lexicographic order of the vertices of $P_{n}^{d}$. This asymptotic comparison was our original motivation to study the bandwidth of the product of graphs. The lex order was proposed to us by a colleague as a labeling to approximate the bandwidth for the discrete Laplacian. It turns out that the lex order is much worse than the Hales order. This asymptotic comparison is carried out  in Section~\ref{sec:lex}.

Theorem~\ref{thm:main} and its proof are extensions of~\cite[Theorem 2]{WWD09} and its proof.

For the rest of the paper, $n$ will denote an arbitrary fixed integer larger than one.

\section{Hales order on $P_{n}^{d}$}\label{sec:P_n^d}

The \emph{Hales order} lists the elements of $V(P_{n}^{d})$ as follows: $u\leq v$ if and only if (i) $w(u)<w(v)$ or (ii) if $w(u)=w(v)$ and $u$ is greater than or equal to $v$ in lexicographic order relative to the right-to-left order of the coordinates. Here, $w(x)$ denotes the \emph{weight} of $x\in V(Q_{d})$ defined as the sum of the components of $x$.  Hales order is often referred to as \emph{graded reverse lexicographic} order.

For instance, the vertices of $P_{2}^{2}=\{0,1,2\} ^{2}$ are listed as follows: $00<01<10<02<11<20<12<21<22$. The order can be obtained by sweeping the hyperplane $y+x=0$ from $(0,0)$ to $(2,2)$, as depicted in Figure~\ref{fig:hales_3^2}. Each time the hyperplane intersects vertices of $P_{2}^{2}$, they are listed according to the direction indicated in the figure.

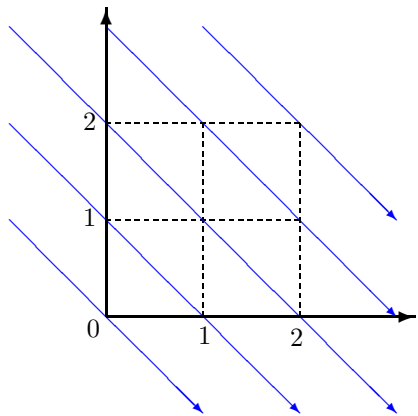
\begin{figure}[h]
\label{fig:hales_3^2}
\begin{center}
\setlength{\unitlength}{2000sp}%
\begingroup\makeatletter\ifx\SetFigFont\undefined%
\gdef\SetFigFont#1#2#3#4#5{%
  \reset@font\fontsize{#1}{#2pt}%
  \fontfamily{#3}\fontseries{#4}\fontshape{#5}%
  \selectfont}%
\fi\endgroup%
\begin{picture}(5059,5059)(3589,-4573)
\put(4514,-1051){\makebox(0,0)[lb]{\smash{{\SetFigFont{10}{16.8}{\familydefault}{\mddefault}{\updefault}{\color[rgb]{0,0,0}2}%
}}}}
\thinlines
{\color[rgb]{0,0,0}\multiput(4801,-2161)(117.07317,0.00000){21}{\line( 1, 0){ 58.537}}
}%
{\color[rgb]{0,0,0}\multiput(6001,-961)(0.00000,-117.07317){21}{\line( 0,-1){ 58.537}}
}%
{\color[rgb]{0,0,1}\put(6001,239){\vector( 1,-1){2400}}
}%
{\color[rgb]{0,0,1}\put(4801,239){\vector( 1,-1){3600}}
}%
{\color[rgb]{0,0,1}\put(3601,239){\vector( 1,-1){4800}}
}%
{\color[rgb]{0,0,1}\put(3601,-961){\vector( 1,-1){3600}}
}%
{\color[rgb]{0,0,1}\put(3601,-2161){\vector( 1,-1){2400}}
}%
\thicklines
{\color[rgb]{0,0,0}\put(4801,-3361){\vector( 0, 1){3825}}
\put(4801,-3361){\vector( 1, 0){3825}}
}%
\put(5941,-3706){\makebox(0,0)[lb]{\smash{{\SetFigFont{10}{16.8}{\familydefault}{\mddefault}{\updefault}{\color[rgb]{0,0,0}1}%
}}}}
\put(7081,-3721){\makebox(0,0)[lb]{\smash{{\SetFigFont{10}{16.8}{\familydefault}{\mddefault}{\updefault}{\color[rgb]{0,0,0}2}%
}}}}
\put(4516,-2236){\makebox(0,0)[lb]{\smash{{\SetFigFont{10}{16.8}{\familydefault}{\mddefault}{\updefault}{\color[rgb]{0,0,0}1}%
}}}}
\put(4561,-3616){\makebox(0,0)[lb]{\smash{{\SetFigFont{10}{16.8}{\familydefault}{\mddefault}{\updefault}{\color[rgb]{0,0,0}0}%
}}}}
\thinlines
{\color[rgb]{0,0,0}\multiput(4801,-3361)(117.07317,0.00000){21}{\line( 1, 0){ 58.537}}
\multiput(7201,-3361)(0.00000,117.07317){21}{\line( 0, 1){ 58.537}}
\multiput(7201,-961)(-117.07317,0.00000){21}{\line(-1, 0){ 58.537}}
\multiput(4801,-961)(0.00000,-117.07317){21}{\line( 0,-1){ 58.537}}
}%
\end{picture}%
\end{center}
\caption{Hales order on $P_{2}^{2}=\{0,1,2\}^{2}$}
\end{figure}

\subsection{Recurrence nature of Hales order}\label{subsec:hales}
The Hales order lists the elements by weight and then by reverse lexicographic order from right-to-left (so $2<1<0$). 
The Hales order exhibits a recursive definition, which we now describe. For $0\leq k\leq nd$ define the $(n+1)^{d}\times d$ matrix $H_{n}^{d}$ with entries in $\{0,1,\ldots,n\}$ as follows:

\begin{equation}\label{eq:H_{d}}
H_{n}^{d}=\begin{bmatrix}
A_{0}^{(n,d)}\\
A_{1}^{(n,d)}\\
A_{2}^{(n,d)}\\
\vdots\\
A_{nd}^{(n,d)}
\end{bmatrix},
\end{equation}
where the $A_{i}^{(n,d)}$ are defined below. Here $\mathbf{k}$ denotes the column vector whose components are all $k$, with $0\leq k\leq nd$, and we let $A_{i}^{(n,1)}=[i]$ for $0\leq i\leq n$.  For $d\geq 2$, define

1) $A_{0}^{(n,d)}=\begin{bmatrix}0&0&\cdots&0 \end{bmatrix}$,

2) for $1\leq k\leq n-1$,
\[
A_{k}^{(n,d)}=\begin{bmatrix}
A_{0}^{(n,d-1)}&\mathbf{k}\\
A_{1}^{(n,d-1)}&\mathbf{k-1}\\
\vdots&\vdots\\
A_{k}^{(n,d-1)}&\mathbf{0}
\end{bmatrix},
\]

3) for $n\leq k\leq nd-n$, 
\[
A_{k}^{(n,d)}=\begin{bmatrix}
A_{k-n}^{(n,d-1)}&\mathbf{n}\\
A_{k-n+1}^{(n,d-1)}&\mathbf{n-1}\\
\vdots&\vdots\\
A_{k}^{(n,d-1)}&\mathbf{0}
\end{bmatrix},
\]

4) for $nd-n+1\leq k\leq nd-1$,
\[A_{k}^{(n,d)}=
\begin{bmatrix}
A_{k-n}^{(n,d-1)}&\mathbf{n}\\
A_{k-n+1}^{(n,d-1)}&\mathbf{n-1}\\
\vdots&\vdots\\
A_{nd-n}^{(n,d-1)}&\mathbf{k+n-nd}
\end{bmatrix}, \text{ and }
\]	

5) $A_{nd}^{(n,d)}=\begin{bmatrix}n&n&\cdots&n\end{bmatrix}$.

For instance, 

\begin{align*}
A_{0}^{(2,3)}&=
\begin{bmatrix}
0&0&0
\end{bmatrix},\;\;
A^{(2,3)}_{1}=
\begin{bmatrix}
0&0&1\\
0&1&0\\
1&0&0
\end{bmatrix},\;\; 
A_{2}^{(2,3)}=
\begin{bmatrix}
0&0&2\\
0&1&1\\
1&0&1\\
0&2&0\\
1&1&0\\
2&0&0
\end{bmatrix},\;\;
A_{3}^{(2,3)}=
\begin{bmatrix}
0&1&2\\
1&0&2\\
0&2&1\\
1&1&1\\
2&0&1\\
1&2&0\\
2&1&0
\end{bmatrix},\\
A_{4}^{(2,3)}&=\begin{bmatrix}
0&2&2\\
1&1&2\\
2&0&2\\
1&2&1\\
2&1&1\\
2&2&0
\end{bmatrix},\;\;
A_{5}^{(2,3)}=
\begin{bmatrix}
1&2&2\\
2&1&2\\
2&2&1
\end{bmatrix},\text{ and}\;\;
A_{6}^{(2,3)}=
\begin{bmatrix}
2&2&2
\end{bmatrix}.
\end{align*}

Notice that the number of rows in $A_{k}^{(n,d)}$ is the number of terms of degree $k$ in the expansion of $(1+x+\cdots+x^{n})^{d}$.

The following theorem relates the order of the rows of $H_{n}^{d}$, from top to bottom,  to the Hales order on $P_{n}^{d}$.
\begin{thm}\label{thm:halesorder}
The rows of $H_{n}^{d}$ listed from top to bottom orders the vertices of  $P^{d}_{n}$ in increasing Hales order.
\end{thm}

\begin{proof}
We proceed by induction on $d$. As in the proof of~\cite[Theorem 1]{WWD09}, it is enough to show that the rows of $A_{k}^{(n,d)}$ are all distinct, have weight $k$ and are sorted, from top to bottom, in decreasing lexicographic order from right-to-left.

The case $d=1$ follows trivially. Let us assume that the result holds for all integers at most $d-1$. Notice that all the rows of $A_{k}^{(n,d)}$ are different from each other, since, by induction hypothesis, the rows of $H^{d-1}_{n}$ were. Furthermore, the weight of the rows of $A_{k}^{(n,d)}$ is $k$ since the rows of $A_{m}^{(n,d-1)}$ have weight $m$. Finally, notice that the vectors of $A_{k}^{(n,d)}$ are listed, from top to bottom, in decreasing lexicographic order from right-to-left, as they share the leftmost $d-1$ components with vectors in $H^{d-1}_{n}$.
\end{proof}

By abuse of notation, we denote the labeling given by the Hales order on $P_{n}^{d}$ by $H_{n}^{d}$.  In particular, $H_{2}^{2}(11)=5$. The Hales order gives the bandwidth of $P_{n}^{d}$, that is,

\begin{prop}[Corollary 3.3, \cite{M05}] For $d\geq0$,
$bw(P_{n}^{d})=bw(H_{n}^{d})$.
\end{prop}

The matrix $H_{n}^{d}$ allows us to find $bw(P_{n}^{d})$ by counting the maximum number of rows between two vertices of $P_{n}^{d}$. So we let $R(A_{k}^{(n,d)})$ denote the number of rows of $A_{k}^{(n,d)}$. In particular,  notice that  $R(A_{k}^{(n,d)})$ is given by the coefficient of the term of degree $k$ in the expansion $(1+x+x^{2}+\cdots + x^{n})^{d}$. That is,
\begin{equation}\label{coeff}
\begin{split}
R(A_{k}^{(n,d)}) &= [x^{k}](1+x+x^{2}+\cdots + x^{n})^{d} \\
    &= \sum_{k_{0},k_{1},\dots,k_{n}} \binom{d}{k_{0}, k_{1}, \dots, k_{n}}
    \end{split}
\end{equation}
where the sum is over $k_{0},k_{1},\dots,k_{n}$ satisfying $\sum k_{i} = d, \sum i k_{i} =k$.
Let $R_{n,d}$ be the array containing the $nd+1$ coefficients of $(1+x+x^{2}+\cdots + x^{n})^{d}$ sorted in decreasing order, that is, $R_{n,d}[1]\geq R_{n,d}[2]\geq\cdots \geq R_{n,d}[nd+1]$ 
and let
\[
\R(n,d)\stackrel{\text{def}}{=}\sum_{i=1}^{n}R_{n,d}[i],
\]
that is, $\R(n,d)$ is the sum of the $n$ largest coefficients of $(1+x+\cdots+x^{n})^{d}$.

We now present the main result of this note.

\begin{thm}\label{thm:main} For $d\geq1$,
\[
bw(H_{n}^{d})=\sum_{i=0}^{d-1}\R(n,i).
\]
\end{thm}

\begin{proof}

Let $(u,v)\in E(P_{n}^{d})$ and assume $w(v) \ge w(u)$ so $w(v)-w(u)=1$, that is,  $u\in A_{k}^{(n,d)}$ and $v\in A_{k+1}^{(n,d)}$ for some $k$. Then either (i) $u$ and $v$ do not share the $d^{th}$ coordinate or (ii) they do.

Case (i) turns out not to be relevant in the optimization problem. Let $u\in[A^{(n,d-1)}_{k-b}\;\;\mathbf{b}]$ and $v\in[A^{(n,d-1)}_{k-b}\;\;\mathbf{b+1}]$, $0\leq b<n$. If $x_{1:m}$ denotes the first $m$ components of a vector, then $u_{1:d-1}=v_{1:d-1}$. Thus there exists $v'\in [A^{(n,d-1)}_{k-b+1}\;\;\textbf{b}]$ satisfying $H_{n}^{d}(v)-H_{n}^{d}(u)<H_{n}^{d}(v')-H_{n}^{d}(u)$ and $(u,v')\in P_{n}^{d}$.

So let us assume that $u\in [A_{k-b}^{(n,d-1)}\;\;\textbf{b}]$ and $v\in[A_{k-b+1}^{(n,d-1)}\;\;\textbf{b}]$, $0\leq b\leq n$. By Theorem~\ref{thm:halesorder}, we have
\[
H_{n}^{d}(v)-H_{n}^{d}(u)=H_{n}^{d-1}(v_{1:d-1})-H_{n}^{d-1}(u_{1:d-1})+R_{n}^{d}(u,v,b),
\]
where $R_{n}^{d}(u,v,b)$ is the number of rows between $u$ and $v$ in $H_{n}^{d}$ that do not have $b$ as last component. Note that each row counted by $R_{n}^{d}(u,v,b)$ is in a matrix $[A_{k-h}^{(n,d-1)}\;\;\textbf{h}]$ with $h\in\{0,1,\ldots,n\}\setminus\{b\}$. There are $n$ such matrices, and since the number of rows of $A_{k-h}^{(n,d-1)}$ is $[x^{k-h}](1+x+\cdots+x^{n})^{d-1}$, it follows that
\[
bw(H_{n}^{d})\leq bw(H_{n}^{d-1})+\R(n,d-1).
\]

To verify that the inequality is tight, set $b=\lfloor n/2\rfloor$, $c=\lfloor n(d-1)/2\rfloor$ and let $u\in [A^{n,d-1}_{c}\;\;\textbf{b}]$ and $v\in [A^{n,d-1}_{c+1}\;\;\textbf{b}]$, with $H_{n}^{d-1}(v_{1:d-1})-H_{n}^{d-1}(u_{1:d-1})=bw(H_{n}^{d-1})$. Notice that the number of rows in $H_{n}^{d}$ between $u$ and $v$ that have $b$ as last component is $bw(H_{n}^{d-1})$. Furthermore, by choice of $c$, the rows counted by $R_{n}^{d}(u,v,b)$ correspond to the first $n$ elements of the array $R_{n,d-1}$, for they correspond to the ``middle'' $n$ coefficients of the expansion of $(1+x+\cdots+x^{n})^{d-1}$ (see Remark~\ref{rem:peak}), that is, $R_{n}^{d}(u,v,b)=\R(n,d-1)$.
\end{proof}

\begin{rem}

If $n=1$, then $\displaystyle\R(n,i)=\binom{i}{\lfloor\frac{i}{2}\rfloor}$ and one recovers the result of Harper~\cite[Corollary 4.4]{H04}. Furthermore, $bw(H_{n}^{2})=\R(n,0)+\R(n,1)=1+n$, and one obtains a special case of Chv\'atalov\'a's~\cite[Theorem 1.1]{C75} stating that $bw(P^{1}_{n}\times P^{1}_{m})=\min\{n+1,m+1\}$.
\end{rem}

\subsection{The case $n=2$}

One can compute the bandwidth of $P_{2}^{d}$ utilizing the following lemma, whose proof follows directly from \eqref{coeff}.

\begin{lem}
$R(A_{k}^{(2,d)})=\sum_{\ell=0}^{\lfloor k/2\rfloor}\binom{d}{d-k+\ell,k-2\ell,\ell}.$
\end{lem}

Table~\ref{tab:values}  exhibits the value of $bw(P^{d}_{n})$ for $0\leq n\leq 8$ and $1\leq d\leq 11$.

\begin{table}[htdp]
\caption{$bw(P^{d}_{n})$ for $1\leq d\leq 11$, $1\leq n\leq 8$}
\begin{center}
\begin{tabular}{|c|c|c|c|c|c|c|c|c|c|c|c|c|c}
\hline
$d$/$n$&1&2&3&4&5&6&7&8\\
\hline 
1&1&1&1&1&1&1&1&1\\
2&2&3&4&5&6&7&8&9\\
3&3&8&14&21&30&40&52&65\\
4&6&21&48&91&155&243&360&509\\
5&12&56&172&404&831&1514&2574&4085\\
6&22&152&617&1835&4512&9655&18716&33551\\
7&42&419&2289&8464&25098&62474&138816&279441\\
8&77&1169&8463&39489&140059&408667&1035692&2352135\\
9&147&3292&32011&185814&793765&2695090&7823236&19956152\\
10&273&9338&120439&880174&4499506&17887694&59241709&170376339\\
11&525&26641&460813&4191494&25788102&119335481&452484637&1461956288\\
\hline
\end{tabular}
\end{center}
\label{tab:values}
\end{table}%

\section{Bounds involving multinomial coefficients} 

In this section we show that $bw(P_{n}^{d})$ is bounded below and above by multinomial coefficients. We first discuss some combinatorial background. 

\subsection{Extended Pascal's triangles} Let $C_{n}(d,k)$ denote the coefficient of $x^{k}$ in the expansion of $(1+x+\cdots+x^{n})^{d}$.

The array
\begin{align*}
&C_{n}(0,0)\\
&C_{n}(1,0)\quad C_{n}(1,1)\quad\cdots\quad C_{n}(1,n)\\
&C_{n}(2,0)\quad C_{n}(2,1)\quad\cdots\quad C_{n}(2,2n)\\
&\hspace{.8in}\quad\vdots\\
&C_{n}(d,0)\quad C_{n}(d,1)\quad\cdots\quad C_{n}(d,dn).
\end{align*}
is called an \emph{extended Pascal's triangle} (or \emph{generalized} Pascal's triangle), and we denote it by $T^{d}_{n}$. That is, $T_{n}^{d}$ is the array where the element in the $(\ell+1)^{st}$ row and $(k+1)^{st}$ column is $C_{n}(\ell,k)$, with $0\leq\ell\leq d$ and $0\leq k\leq \ell n$. These extended Pascal's triangles satisfy  properties akin to those of Pascal's triangle. We list the properties that will be utilized in the sequel.

\begin{lem}\label{lem:multinomial} Let $d,n,k$ be integers with $0\leq k\leq d$. Then,
\begin{enumerate}
\item[(i)] $\displaystyle C_{n}(d,k)=\sum_{j=0}^{n}C_{n}(d-1,k-j)$, where $C_{n}(a,b)=0$ if $b<0$.   \emph{(See~\cite[Equation (2)]{B93}.)}\\
\item[(ii)] $C_{n}(d,k)=C_{n}(d,nd-k)$, that is, each row of $T_{n}^{d}$ is symmetric. \emph{(See~\cite[Equation (4)]{B93}.)}\\
\item[(iii)] The sequence $\{C_{n}(d,k)\}_{k=0}^{nd}$ is log-concave. In particular, each row of $T_{n}^{d}$ is unimodal. \emph{(See~\cite[Theorem 7, Corollary 8]{BS08}.)}
\end{enumerate}
\end{lem}

We refer to Lemma~\ref{lem:multinomial}(i) as the \emph{extended Pascal's identity}. Some remarks are in order.

\begin{rem}\label{rem:peak} Combining the second and third part of the lemma, we conclude that the sequence  $\{C_{n}(d,k)\}_{k=0}^{d}$ is unimodal and has a peak at $k=\lfloor (nd+1)/2\rfloor$; furthermore, just as in the case $n=1$, this sequence can have at most two peaks. The lemma also implies that the largest $i$ elements of each row of $T_{n}^{d}$ are the ``middle'' $i$ elements in each row, that is, the elements located in positions $j$, where  
\[\begin{cases}
\lfloor (nd+1)/2\rfloor-\lfloor i/2\rfloor\leq j\leq \lfloor{(nd+1)/2}\rfloor+\lfloor i/2\rfloor&\text{ if $d$ is even and $i$ is odd,}\\
\lfloor (nd+1)/2\rfloor- i/2\leq j\leq \lfloor{(nd+1)/2}\rfloor+ i/2-1&\text{ if $d$ is even and $i$ is even,}\\
\lfloor (nd+1)/2\rfloor-\lfloor i/2\rfloor\leq j\leq \lfloor{(nd+1)/2}\rfloor+\lfloor i/2\rfloor&\text{ if $d$ is odd and $i$ is odd, and}\\
\lfloor (nd+1)/2\rfloor- i/2\leq j\leq \lfloor{(nd+1)/2}\rfloor+ i/2-1&\text{ if $d$ is odd and $i$ is even.}\\
\end{cases}
\]
\end{rem}

\subsection{The bounds}

Using the notation of Section~\ref{subsec:hales}, we  write $R(A_{k}^{(n,d)})$ instead of $C_{n}(d,k)$. We furthermore denote the highest coefficient in the expansion of $(1+x+\cdots+x^{n})^{d}$ by $M(n,d)$. 
From Remark~\ref{rem:peak}, the first $i$ elements of $R_{n,d}$ are those that correspond to the ``middle'' $i$ degrees of the expansion of $(1+x+\cdots+x^{n})^{d}$. Therefore, the extended Pascal's identity yields
\begin{equation}\label{eq:central}
M(n,d)= \sum_{i=1}^{n+1}R_{n,d-1}[i]=\R(n,d-1)+R_{n,d-1}[n+1].
\end{equation}

Now we prove that $bw(P_{n}^{d})$ is bounded below and above by central multinomial coefficients. The key element in the proof is Equation~(\ref{eq:central}).

\begin{lem}\label{lem:upperbound}
For $n,d\geq0$, $M(n,d)\leq bw(P_{n}^{d})\leq M(n,d+1)$.
\end{lem}

\begin{proof}
We prove both inequalities by induction. The base case $d=1$ is easily verified to hold. Let us suppose that $M(n,d-1)\leq bw(P_{n}^{d-1})\leq M(n,d)$. Then,
\begin{align*}
M(n,d)&
=\R(n,d-1)+R_{n,d-1}[n+1]\\
&\leq \R(n,d-1)+M(n,d-1)\\
&\leq \R(n,d-1)+bw(P_{n}^{d-1})\\
&= bw(P_{n}^{d}).
\end{align*}

For the upper bound, notice that
\begin{align*}
bw(P_{n}^{d})&=bw(P_{n}^{d-1})+\R(n,d-1)\\
&\leq M(n,d)+\R(n,d-1)\\
&= M(n,d)+\sum_{i=1}^{n}R_{n,d-1}[i]\\
&\leq M(n,d)+\sum_{i=2}^{n+1}R_{n,d}[i]\\
&= M(n,d+1),
\end{align*} as desired. We used the fact that $\sum_{i=1}^{n}R_{n,d-1}[i]\leq\sum_{i=2}^{n+1}R_{n,d}[i]$, for the elements of $R_{n,d}$ are sums of elements of $R_{n,d-1}$ by the extended Pascal's identity.
\end{proof}

\subsection{Comparison with the lex order}\label{sec:lex}

Our interest in the bandwidth problem started when a colleague, John Hubbard, asked whether the lexicographic order provided an asymptotically good approximation to the bandwidth for the discrete Laplacian. We point out that one can obtain information regarding the bandwidth of a graph via the discrete Laplacian. For example the Laplacian spectrum of a graph $G$ can be used to give a lower bound on the bandwidth of $G$ (see, e.g.,~\cite{BS208,JM93}). In this section,  we compare the asymptotic behavior of the lexicographic order (from left-to-right) and the Hales order. 

Let $L_{n}^{d}$ denote the label given by ordering the vertices of $P_{n}^{d}$ lexicographically (from left-to-right). Then it can be proved inductively that
\[
bw(L_{n}^{d})=(n+1)^{d-1}.
\] 

For the rest of the section, we consider $n$ to be fixed. We will show the following proposition.\begin{prop}\label{thm:smallo}
$bw(H_{n}^{d})$ is $o(bw(L_{n}^{d}))$ as $d\to\infty$.
\end{prop}
\begin{proof}
We can think of $x^{0},x^{1},\ldots,x^{n}$ as independent random variables uniformly distributed on $\{0,1,\ldots,n\}$. The expected value $\E(x^{i})$ of each $x^{i}$, $0\leq i\leq n$, is $\frac{n}{2}$ and so

the variance $\sigma^{2}(x^{i})$ of each $x^{i}$ is 
\begin{equation*}
\E((x^{i})^{2})-(\E(x^{i}))^{2}=\frac{1}{n+1}\sum_{i=0}^{n}i^{2}-\frac{n^{2}}{4}
 =\frac{n^{2}+2n}{12}.
\end{equation*}

Hence the mean and variance of the $(d+1)$-fold convolution are $n(d+1)/2$ and $(d+1)(n^{2}+2n)/12$, respectively. By the central limit theorem, the coefficients of 
\[
p(d):=\left(\frac{x_{0}+\cdots+x^{n}}{n+1}\right)^{d+1}
\]
are (asymptotically) normally distributed. Furthermore, since the coefficients in the expansion of $p(d)$ are log-concave,~\cite[Lemma 2]{B73} gives that these coefficients satisfy a \emph{local} central limit theorem around the coefficient $M(n,d+1)/(n+1)^{d+1}$. That is,
\[
\frac{M(n,d+1)}{(n+1)^{d+1}}\sim\frac{1}{\sqrt{2\pi}}\sigma^{-1},
\]
as $1/\sqrt{2\pi}$ is the largest value of the density function of the standard normal distribution (hence we obtain an overestimate). Therefore,
\[
M(n,d+1)\sim(n+1)^{d+1}\sqrt{\frac{6}{\pi(d+1)(n^{2}+2n)}}.
\]

Therefore, 
\begin{equation*}
\lim_{d\to\infty}\frac{bw(H_{n}^{d})}{bw(L_{n}^{d})}\leq\lim_{d\to\infty}\frac{M(n,d+1)}{(n+1)^{d-1}}\sim(n+1)^{2}\sqrt{\frac{6}{\pi(d+1)(n^{2}+2n)}}
\end{equation*}
which tends to zero as $d\to\infty$.
\end{proof}

\section{Further directions}
It would be interesting to obtain some kind of bound for the bandwidth of the product of paths of different length. Moghadam~\cite{M05} showed that the Hales order is optimal for these graphs as well, so finding a recursive procedure to gives the Hales order for these graphs might prove helpful to obtain bounds. 

Finally, it would also be interesting to obtain an explicit formula for the \emph{antibandwidth} of the product of paths. The antibandwidth of a graph $G=(V,E)$ is given by 
\[
abw(G)=\max_{f}\min_{u,v\in E}|f(u)-f(v)|.
\]  That is, $abw(G)$ maximizes the minimum absolute value of the difference of labels of any adjacent pair of vertices, over all labelings of $G$. In~\cite{WWD09}, the authors provide an explicit formula for $abw(P_{1}^{d})$, but to the best of our knowledge no such formula exists for the product of paths $P_{n}^{d}$, $n\geq 2$.

\end{document}